\documentclass{amsart}
\usepackage{anysize}
\usepackage{graphicx}
\usepackage{multirow}
\usepackage{fullpage}
\usepackage{amsmath,amssymb,amsfonts}
\usepackage{amsthm}
\usepackage{mathrsfs}
\usepackage{multicol}
\usepackage{verbatim}
\usepackage{url}
\usepackage{enumitem}
\usepackage{mathtools}
\usepackage[dvipsnames]{xcolor}

\makeatletter
\@namedef{subjclassname@2020}{\textup{2020} Mathematics Subject Classification}
\makeatother

\newtheorem{theorem}{Theorem}[section]
\newtheorem{prop}[theorem]{Proposition}
\newtheorem{lemma}[theorem]{Lemma}
\newtheorem{cor}[theorem]{Corollary}

\newtheorem{quest}[theorem]{Question}

\newtheorem{remark}[theorem]{Remark}

\DeclareMathOperator{\homeo}{Homeo}

\DeclareMathOperator{\diff}{Diff}
\DeclareMathOperator{\thomeo}{Homeo^{1}_1}

\DeclareMathOperator{\supp}{supp}
\DeclareMathOperator{\id}{id}

\setlist{noitemsep}
\marginsize{1in}{1in}{1in}{1in}
\begin{document}

\title{A local characterization for homeomorphisms that \\preserve continuously differentiable curves}
\author{Katherine Williams Booth}
\address{University of Chicago, Chicago, IL 60637}
\email{k.wbooth3@gmail.com}
\date{}

\begin{abstract} We study the subgroup of the homeomorphism group consisting of homeomorphisms of a manifold that preserve the set of one-dimensional $C^1$-submanifolds. For many surfaces, the similarly-defined family of related groups preserving the set of one-dimensional $C^r$-submanifolds has been shown by the author to be isomorphic to the automorphism groups of smooth fine curve graphs.  In this paper, we give a necessary and sufficient local characterization for elements of this group using properties on the projective tangent bundle. In addition, we give several new examples of elements from this group that are not $C^1$-diffeomorphisms.
\end{abstract}

\maketitle

\section{Introduction}\label{intro}

Given a smooth manifold $M$ of dimension $n$, with $n$ at least 2,  we say that a simple closed curve $\alpha$ in $M$ is $C^1$ if it is a properly embedded one-dimensional $C^1$-submanifold of $M$. Equivalently,  these are the images of continuously differentiable and injective $\gamma: S^1 \rightarrow M$ with $\gamma'(t) \neq 0$ for all $t \in S^1$.  We define the group $\thomeo(M)$ to be the set of $F\in\homeo(M)$ such that both  $F$ and $F^{-1}$ map every $C^1$ curve to a $C^1$ curve.  

By varying the dimension $k$ and regularity $r$ of the preserved submanifolds, we can similarly define a family $\homeo^k_r(M)$ of subgroups of $\homeo(M)$. In a separate paper \cite{BoothAutck}, we have shown that when $M$ is a compact orientable surface of genus at least 2,  $\homeo^1_r(M)$ is naturally isomorphic to the automorphisms of the $C^r$-curve graph.  While these subgroups are naturally-defined, they do not appear in the prior literature apart from one example from $\homeo^1_{\infty}(M)$ given by Le Roux--Wolff~\cite{LRW}. By slightly amending their example, it can be seen that the group of $C^r$-diffeomorphisms of $M$,  denoted $\diff^r(M)$ is a proper subgroup of $\homeo^k_r(M)$ when $r \geq 1$ and $1 \leq k \leq n-1$. Since there also exist homeomorphisms that do not preserve $C^r$-submanifolds, we have the following chain of proper inclusions:
$$\diff^r(M)\; \lneq \;  \homeo^k_r(M) \;  \lneq \; \homeo(M)$$ 
This family of groups provides a rich area of potential study, with potentially unexpected results. By increasing either dimension or regularity, the rigidity and structure on the submanifolds increases. But at the same time, the number of submanifolds that can potentially constrain the behavior of the homeomorphism decreases. Even for the simplest case of $\thomeo(S)$ for a surface $S$, we describe several new examples in Section~\ref{examples} that exhibit increasingly interesting and unexpected behavior.

The remainder of this introduction is organized as follows.  In Section~\ref{intromainthm}, we define relevant structures and state our main theorem.  We also give brief descriptions of our examples,  which are defined in more detail in Section~\ref{examples}.   In Section~\ref{introtangentbundle},  we go into a further discussion on the projective tangent bundle of the surface and its connections to $C^1$ curves.  Section~\ref{introquestions} contains a collection of open questions related to the family $\homeo^k_r(M)$.  Finally, in Section~\ref{introproofidea},  we give an overview of the proof of our main theorem and an outline of the paper.

\subsection{Statement of our main theorem}\label{intromainthm}

Before stating our results, we give a couple of necessary definitions.

\medskip

\noindent \textit{Maps on projective tangent spaces.}
One challenge that presents itself when working with $C^1$ curves is that the definition relies on geometrical characteristics instead of the algebraic properties of a specific parameterization. In particular, tangent vectors are not well-defined for $C^1$ curves. To address this, we work instead with the tangent line.  We call the space of tangent lines at a point $p$ the \emph{projective tangent space}, denoted by $\mathbb{P}T_pM$. We note that each $\ell \in \mathbb{P}T_pM$ is associated with two unit tangent vectors $\{\mathbf{u_{\ell}}, -\mathbf{u_{\ell}}\}$ in the unit tangent space $UT_pM$ and that $\mathbb{P}T_pM \cong \mathbb{RP}^{n-1}$. Any homeomorphism $F$ that preserves the equivalence classes of curves with the same tangent line at the point $p$ induces a well-defined map $$\bar{d}F_p: \mathbb{P}T_pM \rightarrow \mathbb{P}T_{F(p)}M$$ of the projective tangent spaces.  If $F$ preserves these equivalence classes at every point in $M$, then there is also a well-defined map on the projective tangent bundle $$\bar{d}F:~\mathbb{P}TM \rightarrow \mathbb{P}TM.$$

\noindent  We refer to both $\bar{d}F_p$ and $\bar{d}F$ as \emph{tangent maps} for $F$. We denote points in $\mathbb{P}TM$ by $(p, \ell)$, where $p \in M$ and $\ell \in \mathbb{P}T_pM$. Any line $\ell \in T_pM$ through the origin is also naturally identified with a point in $\mathbb{P}T_pM$.  When convenient, we abuse notation and denote this line also by $\ell \in \mathbb{P}T_pM$.

\medskip

\noindent \textit{Convergance along lines, strong convergence, and transverse sequences. } Let $\ell$ be a line through the point $p$ in $\mathbb{R}^n$ with a unit tangent vector $\mathbf{u}_{\ell} \in \mathbb{S}^{n-1}$.  A sequence of points $\{p_k\}_k$ converging to $p$ in $\mathbb{R}^n$ is said to \emph{converge to $p$ along $\ell$ (in the positive direction) } if 
$$\lim_{k \rightarrow \infty} \frac{p_k-p}{||p_k-p||} = \mathbf{u_{\ell}}$$


To expand this definition to any manifold, we identify each tangent space $T_pM$ with $\mathbb{R}^n$.   For any line $\ell \in T_pM$ through the origin,  we say that the sequence $\{p_k\}$ in $M$ \emph{converges to $p$ along $\ell$ } if there exists a smooth coordinate chart $\varphi: U \subset M \rightarrow T_pM$ based at $p$ such that $\varphi(p_k)$ converges to $\varphi(p)=0$ along $\ell$.  Such a chart can always be obtained by fixing a Riemannian metric and taking the inverse of an exponential map based at $p$. Moreover, this local chart also allows a direct comparison of lines in nearby projective tangent spaces and a well-defined notion of convergence in $\mathbb{RP}^{n-1}$ for sequences of lines.

We say that a sequence $\{(p_k, \ell_k)\}$ in $\mathbb{P}TM$ \emph{strongly converges to} $(p, \ell)$ if $\{p_k\}$ converges to $p$ along $\ell$ and the sequence $\{\ell_k\}$ converges to $\ell$ in $\mathbb{RP}^{n-1}$. If such a $(p,\ell)$ exists for a sequence $\{(p_k, \ell_k)\}$, then we say that $\{(p_k, \ell_k)\}$ is a \emph{strongly convergent sequence}.

  A \emph{transverse sequence for $(p, \ell)$} is a sequence $\{(p_k, \ell_k)\}$ in $\mathbb{P}TM$ such that $\{p_k\}$  converges to $p$ along $\ell$  and no subsequence of the $\ell_k$ converges to $\ell$ in $\mathbb{RP}^{n-1}$. In other words, no subsequence is itself a strongly convergent sequence.  If such a $(p, \ell)$ exists for a sequence $\{(p_k, \ell_k)\}$, then we call $\{(p_k, \ell_k)\}$ a \emph{transverse sequence}.

\medskip

\noindent We now give the necessary and sufficient properties for a homeomorphism of $M$ to be an element of $\thomeo(M)$.

\begin{theorem}\label{theoremconditions}
Fix a smooth manifold $M$ of dimension at least $2$.  A homeomorphism $F$ of $M$ is an element of $\thomeo(M)$ if and only if $F$ has the following three properties:
\begin{enumerate}[noitemsep,topsep=0pt, label=$(\alph*)$]
\item for every point $p \in M$, there exists a well-defined homeomorphism $\bar{d}F_p$ induced by the action of $F$ on the tangent lines based at $p$,
\item $F$ maps every strongly convergent sequence to a strongly convergent sequence, and
\item $F$ maps every transverse sequence to a transverse sequence.
\end{enumerate}
\end{theorem}

\noindent While $\diff^1(M)$ is a subgroup of $\thomeo(M)$ directly from the definition, it is straightforward to see that any elements of $\diff^1(M)$ also has the three properties stated in Theorem \ref{theoremconditions}. 
For property~$(a)$, the differential maps between the tangent spaces are linear, so they induce homeomorphisms on the projective tangent spaces.  
Properties $(b)$ and $(c)$ follow directly from the fact that elements of $\diff^1(M)$ are continuous and injective on the tangent bundle.  

\medskip

\noindent \textit{New example elements. } In addition to the our main theorem, we also give examples in Section \ref{examples} of several phenomena that appear in elements of $\thomeo(S)$ for surfaces $S$. These are detailed in the following theorem:

\begin{theorem}\label{propexamples}
There exists an element $F \in \thomeo(S){\setminus} \diff^1(S)$ and a sequence of points $p_n \in S$ such that 
\begin{enumerate}[noitemsep,topsep=0pt, label=$(\roman*)$]
\item $\bar{d}F_{p_1}$ is the identity, but $F$ is not differentiable at $p_1$, 
\item the pushforward map $dF_{p_2} : T_{p_2}S \rightarrow T_{f(p_2)}S$ is well-defined and maps every element in $T_{p_2}S$ to~$0$, 
\item the map $\bar{d}F$ on the projective tangent bundle is not continuous, and 
\item $F$ is not differentiable at $p_n$, for all $n \in \mathbb{N}$.
\end{enumerate}
\end{theorem}
 
 \noindent Our final result highlights the necessity of transverse sequence in the local characterization of $\thomeo(M)$.
 
\begin{theorem}\label{propNex}
There exists an element of $\homeo(S){\setminus} \thomeo(S)$ that has properties $(a)$ and $(b)$ from Theorem~\ref{theoremconditions}.
\end{theorem}

\subsection{Strongly convergent and transverse sequences}\label{introtangentbundle}

In this section, we briefly discuss the projective tangent bundle and another perspective for convergent and transverse sequences.  

\medskip

\noindent \textit{Lifts of $C^1$ curves in the projective tangent bundle and strongly convergent sequences.  }
One distinct characteristic of $C^1$ curves is that they have tangent lines at every point and this tangent line varies continuously.  As a consequence,  every $C^1$ curve uniquely defines a continuous curve in the projective tangent bundle. Even more, any convergent sequence of points in a $C^1$ curve, with their corresponding tangent lines, is a strongly convergent sequence. 

But not all continuous curves in the projective tangent bundle are lifts of $C^1$ curves. For example, we consider the arc $\left\{\left( (t, 0), x=0 \right)\right\}$ in $\mathbb{P}T\mathbb{R}^2$. The points in $\mathbb{R}^2$ move along the horizontal $x$-axis to the origin, but the lines in the projective tangent spaces are vertical at every point. There is a disconnect between the specified tangent directions and the actual direction of movement along which the points are approaching the origin.

\medskip

\noindent \textit{Transverse sequences as obstructions to $C^1$ curves.  } In our context, we utilize transverse sequences as obstructions to the existence of a curve in the projective tangent bundle.  No subsequence of the points $\{(p_k, \ell_k)\}$ in $\mathbb{P}TM$ corresponding to a transverse sequence can converge to the point $(p,\ell)$.  But any differentiable curve in $M$ that passes through infinitely many of the $p_k$ must have a tangent line of $\ell$ at $p$.  So the lift to $\mathbb{P}TM$ of any curve in $M$ that contains a subsequence of $\{(p_k, \ell_k)\}$ is not continuous and cannot be a curve in $\mathbb{P}TM$. 

\begin{figure}[ht]
\centering
\includegraphics[width=119mm]{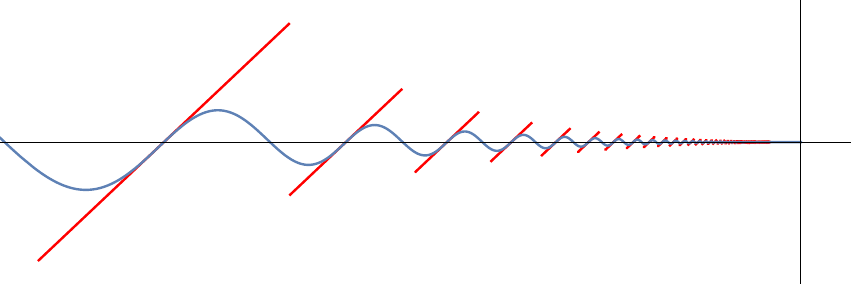}
\caption{A transverse sequence of tangent lines from $y=x^2 \sin(1/x)$ with slope 1 that converge to the origin along the $x$-axis}
\end{figure}

Similarly,  any differentiable curve that is not $C^1$ has tangent lines that do not vary continuously and must contain a transverse sequence.  Thus property $(c)$ from Theorem~\ref{theoremconditions} prevents such a curve from being mapped to a $C^1$ curve.  This prevents the inverse homeomorphism on $M$ from sending any $C^1$ curve to a non-$C^1$ curve.

\subsection{Further questions}\label{introquestions}
The properties of $\thomeo(M)$ from Theorem~\ref{theoremconditions} and the examples in Section~\ref{examples} suggest several possible questions about $\thomeo(M)$ and the family of groups $\homeo^k_r(m)$.  We have listed a few questions of interest here.

\medskip 

\begin{quest} How does the dimension of the submanifolds affect the local characterization of $\homeo^k_1(M)$? Are these groups nested?
\end{quest}

\noindent One possible way to approach the above question is to think of lower dimensional submanifolds of a particular regularity as the intersection of higher dimensional submanifolds. This perspective suggests that the homeomorphisms preserving the higher dimensional submanifolds might also be forced to preserve lower dimensional submanifolds with the same regularity.

\medskip 

\begin{quest} Can the set of non-differentiable points in an element of $\thomeo(S)$ be uncountable? Or have positive measure?
\end{quest}

\noindent One concern with creating such examples is that the non-differentiability of nearby points might interfere with each other to such a level that forces some $C^1$ curves to no longer by $C^1$. The example given in Section~\ref{examples} was built specifically to avoid any interaction between the non-differentiable points, thus avoiding this possibility.

\medskip

The next several questions arise from comparing $\homeo^k_r(M)$ to the groups $\homeo(M)$ and $\diff^r\!(M)$. An overview of many of the related results for these groups can be found in an article by Mann~\cite{Mann}. 

\medskip

\begin{quest} For compact $M$, is the identity component for the groups $\homeo^k_r(M)$ perfect? And simple? \end{quest}

\medskip 

\noindent A group is \emph{perfect} if every element can be written as a product of commutators. For compact manifolds, the identity components $\homeo_0(M)$ and $\diff^r_0(M)$ are known to be perfect for every value of $r$ except $r=\mathrm{dim}(M)+1$. The proof for $\homeo_0(M)$ has been known since the 1950s and can be found in Anderson~\cite{Anderson}. On the other hand, the proof of perfectness for $\diff_0^r(M)$ is significantly more challenging, completed by Thurston ($r= \infty$)~\cite{Thurston} and Mather ($r \neq \mathrm{dim}(M)+1, \neq \infty$)~\cite{Mather1}, \cite{Mather2} in the 1970s. The case for $r=\mathrm{dim}(M)+1$ is still not known.  

In addition, perfectness is one step towards showing that the identity component of $\homeo^k_r(M)$ is \emph{simple}, meaning it has no nontrivial proper normal subgroups. Again, the groups $\homeo_0(M)$ and $\diff^r_0(M)$ for $r \neq \mathrm{dim}(M)+1$ were shown to be simple in the 1970s by Edwards--Kirby~\cite{EK}, Thurston~\cite{Thurston}, and Mather~\cite{Mather1}, \cite{Mather2}. 

\medskip

\begin{quest} Does $\homeo^k_r(M)$ have enough algebraic structure to identify $M$? Does it satisfy automatic continuity?  \end{quest}

\medskip

\noindent This question arises in direct comparison with known properties from $\homeo(M)$ and $\diff(M)$. In the 1960's,  Whittaker~\cite{Whitt} showed that the existence of an isomorphism between the homeomorphism groups of two compact manifolds forces the two manifolds to be homeomorphic. A similar phenomena occurs for diffeomorphism groups due to Filipkiewicz~\cite{Filip}, with the algebraic structure forcing both the groups and the diffeomorphism between the manifolds to have the same regularity. 

A topological group $G$ has \emph{automatic continuity} if every homomorphism from $G$ to a separable group is necessarily continuous.  Rosendal~\cite{Rosendal} and Mann~\cite{Mann2016} have showed that when $M$ is compact, then $\homeo(M)$ has automatic continuity.

\medskip

\begin{quest} Is there a natural topology for $\homeo^k_r(M)$ which makes it Polish? \end{quest}

\noindent  A Polish group is a topological group that is separable and completely metrizable. As remarked in~\cite{BDR}, this additional structure provides a framework for proving automatic continuity, used by both arguments of Rosendal and Mann.  With the compact-open and $C^r$ topologies, both $\homeo(M)$ and $\diff^r(M)$ are Polish groups. In addition, it has been shown that many subgroups of $\homeo(M)$ must have a unique Polish topology~\cite{Kallman}. 

Since each $\homeo^k_r(M)$ is a subgroup of $\homeo(M)$, then these groups inherit the compact-open topology from $\homeo(M)$.  But for $\homeo^k_r(M)$, the inherited topology is not Polish since it does not account for the behavior on the projective tangent bundle. It is plausible that some modification of the $C^r$ topology of $\diff^r(M)$ would be sufficient to find a Polish topology. But such Polish topology is not guaranteed for subgroups of $\homeo(M)$. For example, the group of piecewise linear homeomorphisms cannot have a Polish topology~\cite{Mann2017}.

\subsection{Paper structure and overview of the proof}\label{introproofidea}
This paper is split into four parts.  We start with Section~\ref{examples}, giving several examples proving Theorems \ref{propexamples} and \ref{propNex}. In Section~\ref{section-basics}, we prove a couple of foundational results about $C^1$ curves that are used extensively in the later sections. Specifically, we show that $C^1$ curves can be drawn through subsequences of convergent points and give a characterization of the existence of the tangent map at a point using sequences of points that converge along a line. 

The remainder of the paper is the proof of Theorem~\ref{theoremconditions}.  The proof is split into two parts, each corresponding to one of the logical directions of the statement.  The forward direction involves recovering each of the three properties stated in the theorem from an element of $\thomeo(M)$,  while the reverse direction requires us to show that any map with these properties is an element of $\thomeo(M)$.  

\medskip

\noindent \textit{Forward direction.} We complete this step in Section~\ref{sectionforward}.  
To prove the forward direction, we must have that any element of $\thomeo(M)$  induces well-defined homeomorphisms on the projective tangent space, sends strongly convergent sequences to strongly convergent sequences, and sends transverse sequences to transverse sequences.  

We first need to show that an element of $\thomeo(M)$ induces a well-defined map on $\mathbb{P}T_pM$ for every point $p \in M$.   We prove this by showing that any homeomorphism of $M$ that sends every $C^1$ curve to a $C^1$ curve must respect the equivalence classes at each point of curves with the same tangent line. We then show that the map on equivalence classes must be a bijection by leveraging the fact that the inverse also maps $C^1$ curves to $C^1$ curves. To show that this map is a homeomorphism, we apply a standard continuity argument using convergent sequences in the unit tangent bundle. This argument also utilizes sequences of points and the results from Section~\ref{section-basics}.

To get the second property, we show that if strongly convergent sequences are not preserved, there exists a $C^1$ curve that is mapped to a curve that cannot be $C^1$.

Finally,  to show that an element of $\thomeo(M)$ sends all transverse sequences to transverse sequences, we again work from a contradiction.  We assume that the homeomorphism maps some transverse sequence to a sequence that is not transverse and has a convergent subsequence.  We are then able to construct a $C^1$ curve whose preimage is not $C^1$.

\medskip

\noindent \textit{Reverse direction.}
This step of our proof is completed in Section~\ref{sectionreverse}.  The goal is to show that any homeomorphism satisfying the three properties from Theorem~\ref{theoremconditions} must be an element of $\thomeo(M)$.  The existence of a well-defined map on tangent lines and the preservation of strongly convergent sequences is sufficient to show that $F$ maps every $C^1$ curve to a $C^1$ curve.

To show that $F^{-1}$ also maps every $C^1$ curve to a $C^1$ curve, we use the homeomorphism on the projective tangent space at each point to recover that every image of a $C^1$ curve under $F^{-1}$ must have a unique tangent line at every point.  We then use transverse sequences to show that these tangent lines must vary continuously, which is an equivalent characterization of $C^1$ curves.

\subsection*{Acknowledgments}

We would like to thank our advisor Dan Margalit for his generous support and encouragement.  In particular, we would like to thank him for his frequent suggestions of ways to simplify the description of $\thomeo(M)$ that gave us the push to find increasingly interesting examples.

We thank Alex Nolte for his relentless enthusiasm and interest in this project.  We especially thank him for his contribution to the arguments involving transverse sequences.

We thank anonymous referees for many helpful comments and the suggestion to generalize this work to more dimensions. We thank Igor Belegradek for his interest and questions.  We also thank Noah Caplinger, Sean Eli,  Daniel Minahan,  and Roberta Shapiro for helpful conversations and general encouragement for this project. 

This work was partially funded by the National Science Foundation under Grant Nos. DMS-1745583, DMS-2417920, and DMS-2503473. This work was also supported by the National Science Foundation under Grant No. DMS-2424139, while the author was in residence at the Simons Laufer Mathematical Sciences Institute in Berkeley, California, during the Spring 2026 semester.

\section{Example elements from $\thomeo(S)$} \label{examples}
In this section, we prove Theorems \ref{propexamples} and \ref{propNex}.   We accomplish this by giving several examples from $\thomeo \!\left(\mathbb{R}^2\right){\setminus}\diff^1\!\left(\mathbb{R}^2\right)$.  Many of these examples were discovered to dismiss potentially simpler characterizations for  $\thomeo(M)$.  We give these examples explicitly on $\mathbb{R}^2$, since Theorem~\ref{theoremconditions} only concerns the local conditions around each point.  

\subsection{Descriptions of the examples} \label{subexamples}

In this section,  we start with the examples for parts $(i)$ and $(ii)$ from Theorem~\ref{propexamples}.  We then give a non-example that proves Theorem~\ref{propNex}. Finally, we conclude this section with the examples for parts $(iii)$ and $(iv)$ from Theorem~\ref{propexamples}.

\medskip

\noindent \textit{Example $(i)$: Modifying magnitudes with Le~Roux--Wolff.  } This example proves the following lemma related to part $(i)$ from Theorem~\ref{propexamples}.

\begin{lemma}\label{lemex1}
There exists an element $G \in \thomeo \!\left(\mathbb{R}^2\right){\setminus}\diff^1\!\left(\mathbb{R}^2\right)$ such that $\bar{d}G_{(0,0)}$ is the identity, but $G$ is not differentiable at $(0,0)$. 
\end{lemma}

\begin{proof}
We build $G:\mathbb{R}^2 \rightarrow \mathbb{R}^2$ similarly to an example given by Le Roux--Wolff \cite[Section 5]{LRW}.  In their example, they change the $y$-coordinate based on the slope of the line to the origin.   We modify their construction by modifying both coordinates instead of only one of them.  

Let $g: \mathbb{R} \rightarrow \mathbb{R}^+$ be a smooth map such that $g(1) \neq 1$ and $g(t)=1$ whenever $t$ is outside of the interval $[1/2,2]$.  Define the map $G:\mathbb{R}^2\rightarrow\mathbb{R}^2$ by $$G(x,y) = \left\{ \begin{array}{ll} (x \, g(y/x), y \, g(y/x) ) & \mbox{if } x \neq 0 \\  (x, y) & \mbox{if } x =0 \end{array} \right.$$
Note that $G$ is a $C^1$ diffeomorphism at every point except the origin.  We can show by direct computation that every $C^1$ curve through the origin is mapped by $G$ to a $C^1$ curve.  Since $G^{-1}$ is of the same form as $G$ by replacing $g$ by $1/g$, then $G^{-1}$ also maps every $C^1$ curve to a $C^1$ curve. Thus $G \in \thomeo \!\left(\mathbb{R}^2 \right)$.  

\begin{figure}[ht]
\centering
\includegraphics[width=119mm]{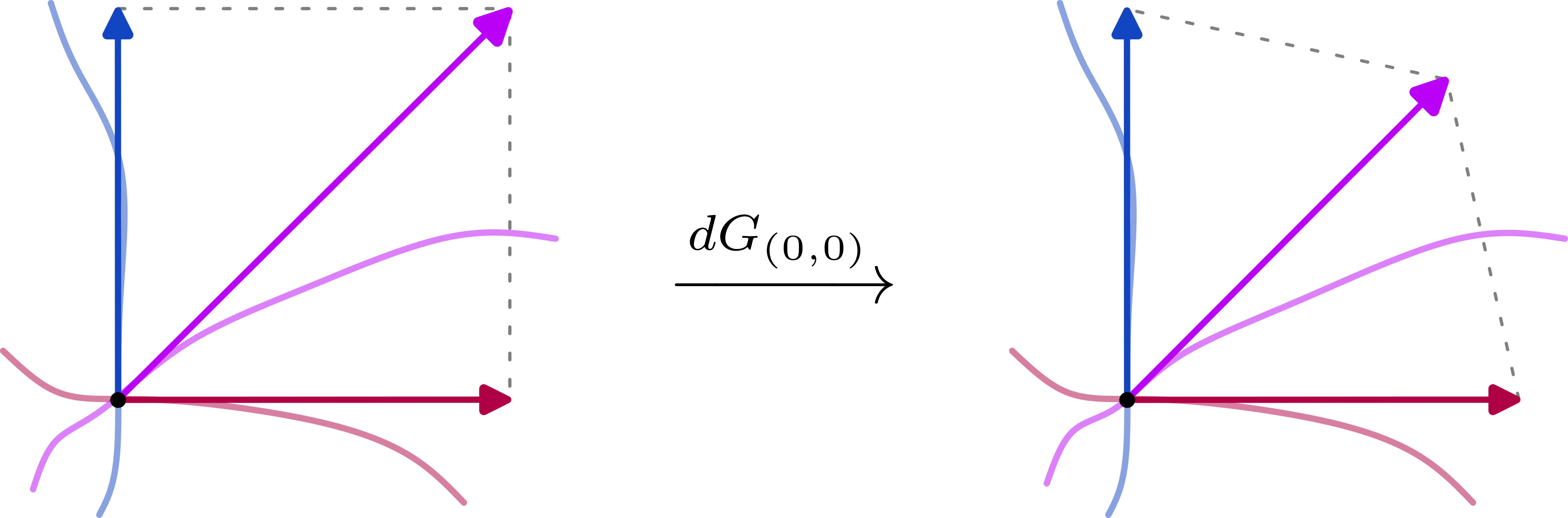} 
\caption{The action of $G$ on the vectors $(1, 0)$, $(0, 1)$, and $(1,1)$ based at the origin}
\end{figure}

To show that $G \notin \diff^1 \! \left( \mathbb{R} ^2 \right)$, we consider the action of $G$ on vectors based at the origin.  By changing the magnitude of the vectors along the line with slope 1, we prevent the pushforward map $dG_{(0,0)}$ from being a linear map.   Specifically, $dG_{(0,0)}(1,0) = (1,0)$ and $dG_{(0,0)}(0,1) = (0,1)$,  while $dG_{(0,0)}(1,1) = (g(1), g(1)) \neq (1,1)$.  

Since both $x$- and $y$-coordinates are changed by the same multiplicative factor,  this map does not change the tangent direction of any $C^1$ curve through the origin.  So the induced map $\bar{d}G_{(0,0)}$ is the identity map on $\mathbb{P}T_{(0,0)}\mathbb{R}^2$.  
\end{proof}

\medskip

\noindent \textit{Example $(ii)$: Collapsing and expanding tangent vectors.  } For part $(ii)$ from Theorem~\ref{propexamples}, this example proves the following:

\begin{lemma}\label{lemex2}
There exists an element $H \in \thomeo \!\left(\mathbb{R}^2\right){\setminus}\diff^1\!\left(\mathbb{R}^2\right)$ fixing the origin such that the pushforward map $dH_{(0,0)} : T_{(0,0)}\mathbb{R}^2 \rightarrow T_{(0,0)}\mathbb{R}^2$ is well-defined and maps every element in $T_{(0,0)}\mathbb{R}^2$ to~$0$.
\end{lemma}

\begin{proof}
Consider the map $H: \mathbb{R}^2 \rightarrow \mathbb{R}^2$ defined in polar coordinates as follows:
\begin{flalign*}
H: (r,  \theta) & \mapsto (sgn(r) \, r^2, \theta) 
\end{flalign*}

\begin{figure}[ht]
\centering
\includegraphics[width=119mm]{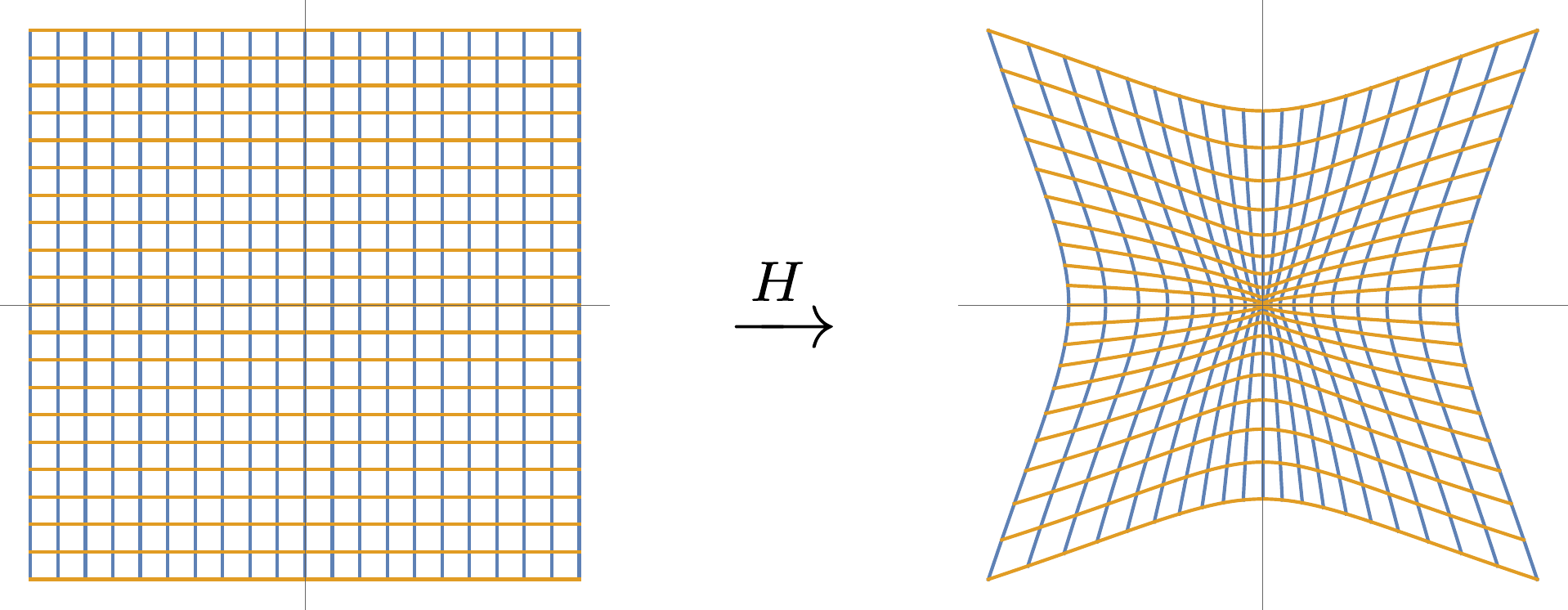}
\caption{The action of $H$ on a grid near the origin}
\end{figure}

For any $C^1$-regular parameterization of a curve $\gamma(t)$ such that $\gamma(0)=(0,0)$, the curve $H \circ \gamma$ has a tangent vector with magnitude $0$ at the origin.  By reparameterizing by $s=sgn(t)\,  t^2$, it can be shown that the image of $H\circ \gamma$ is still a $C^1$ curve. 

The inverse map
\begin{flalign*}
H^{-1}: (r,  \theta) & \mapsto (sgn(r) \, \sqrt{|r|}, \theta) 
\end{flalign*}

\begin{figure}[ht]
\centering
\includegraphics[width=119mm]{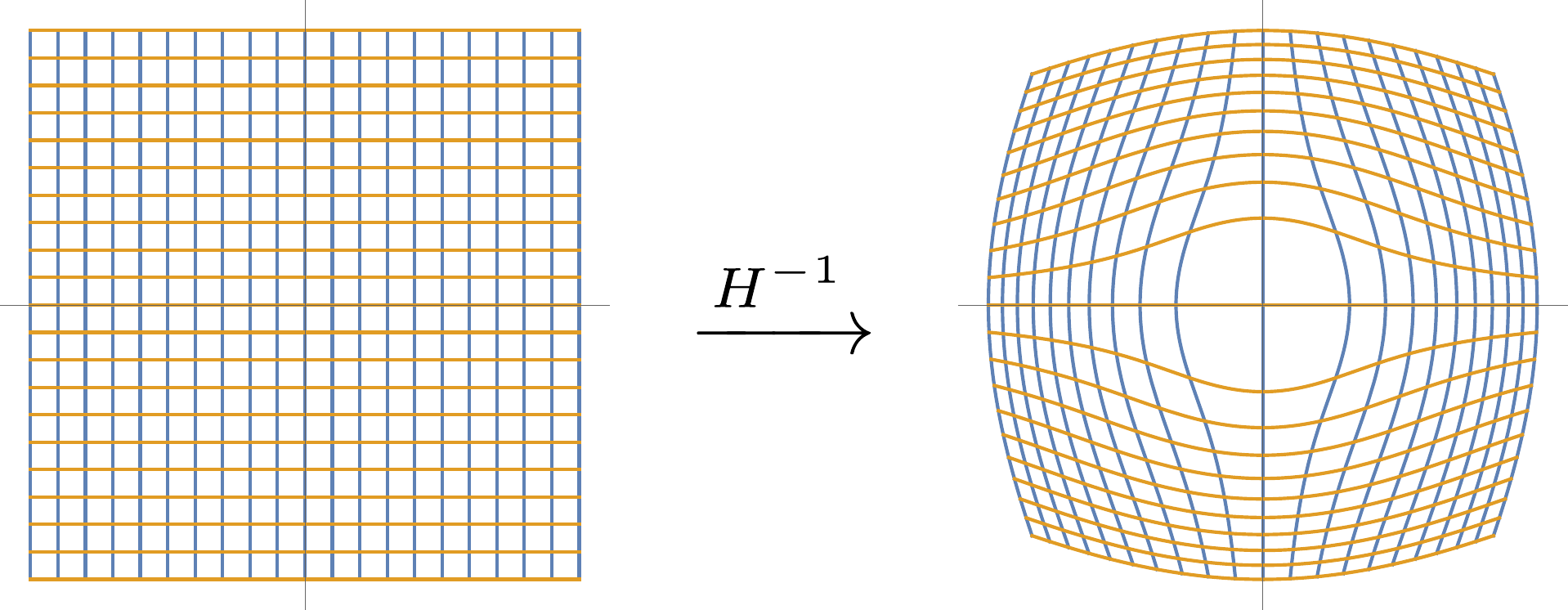}
\caption{The action of $H^{-1}$ on a grid near the origin}
\end{figure}
\noindent acts similarly.  The pushforward map $dH^{-1}_{(0,0)}$ sends every nonzero vector based at the origin to a vector with infinite magnitude.  Again by reparameterizing,  it can be shown that $H^{-1}$ also maps every $C^1$ curve to a $C^1$ curve. Thus $H$ is an element of $\thomeo\! \left(\mathbb{R}^2\right)$.
\end{proof}

The examples of $H$ and $H^{-1}$ both arise from the fact that we are concerned about the image of the curve and not a particular parameterization.  This suggests that the action on $\mathbb{P}TM$ is the appropriate object to consider for elements of $\thomeo(M)$ instead of the pushforward maps on $TM$.

\medskip

\noindent \textit{A cautionary non-example for $\thomeo(S)$.  }
We now take a brief respite from examples of $\thomeo\!\left(\mathbb{R}^2\right)$ to give an element from $\homeo\!\left( \mathbb{R}^2 \right){\setminus}\thomeo\!\left(\mathbb{R}^2\right)$.  The homeomorphism $W$ given below maps every $C^1$ curve to a $C^1$ curve, but its inverse $W^{-1}$ maps some $C^1$ curves to curves that are not $C^1$.  

\begin{proof}[Proof of Theorem~\ref{propNex}]
 To formally define $W$, we start with a homeomorphism $w: [0,1] \rightarrow [0,1]$ such that $w(0)=0$, $w(1)=1$, $w'(0)=0 = w'(1)$ and $w$ is a $C^1$ diffeomorphism on $(0,1)$.  For $\theta \in [0,\pi)$ and $r \in \mathbb{R}$, we then define

$$W(r, \theta) = \left\{ \begin{array}{ll}  \left(r, \pi \, w\left( \theta / \pi \right) \, (1- w(|r|)) + \theta \, w(|r|) \right) & \mbox{if } |r|<1 \\
(r, \theta) & \mbox{if } |r| \geq 1
\end{array} \right. $$

\begin{figure}[ht]
\centering
\includegraphics[width=119mm]{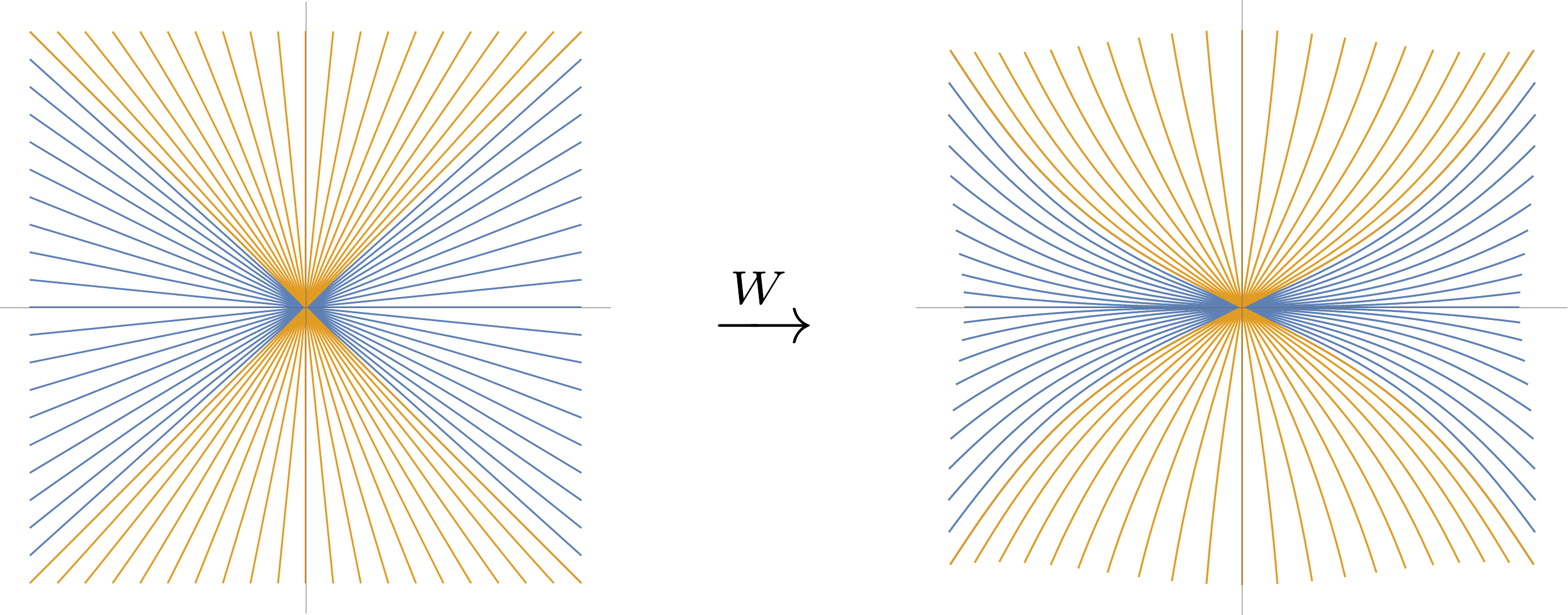}
\caption{The action of $W$ on lines passing through the origin}
\end{figure}

It can be shown that $W$ is a homeomorphism that sends every $C^1$ curve to a $C^1$ curve.  Moreover, it induces well-defined homeomorphisms $\bar{d}W_p$ on the projective tangent spaces at every point $p \in \mathbb{R}^2$.  

But $W^{-1}$ does not map every $C^1$ curve to a $C^1$ curve.  At the origin, $\bar{d}W_{(0,0)}$ compresses the angles near $0$ and $\pi$ in such a way that $W$ sends the curve defined by $y=x^2 \sin (1/x)$ to a $C^1$ curve.  Thus $W$ is not an element of $\thomeo\!\left(\mathbb{R}^2\right)$.
\end{proof}

\medskip

\noindent \textit{Example $(iii)$: Non-continuity on the projective tangent bundle.  }  Recall that elements of $\thomeo(S)$ induce a map on $\mathbb{P}TS$.  By property $(b)$ from our Theorem~\ref{theoremconditions}, we know that this map on $\mathbb{P}TS$ must be a bijection.  But the following result related to part $(iii)$ of Theorem~\ref{propexamples} shows that this induced map may not be continuous.  

\begin{lemma}\label{lemex3}
There exists an element $Q \in \thomeo \!\left(\mathbb{R}^2\right){\setminus}\diff^1\!\left(\mathbb{R}^2\right)$ such that the map $p \mapsto \bar{d}_pQ$ is not continuous.
\end{lemma}

\begin{figure}[h!]
\centering
\includegraphics[width=119mm]{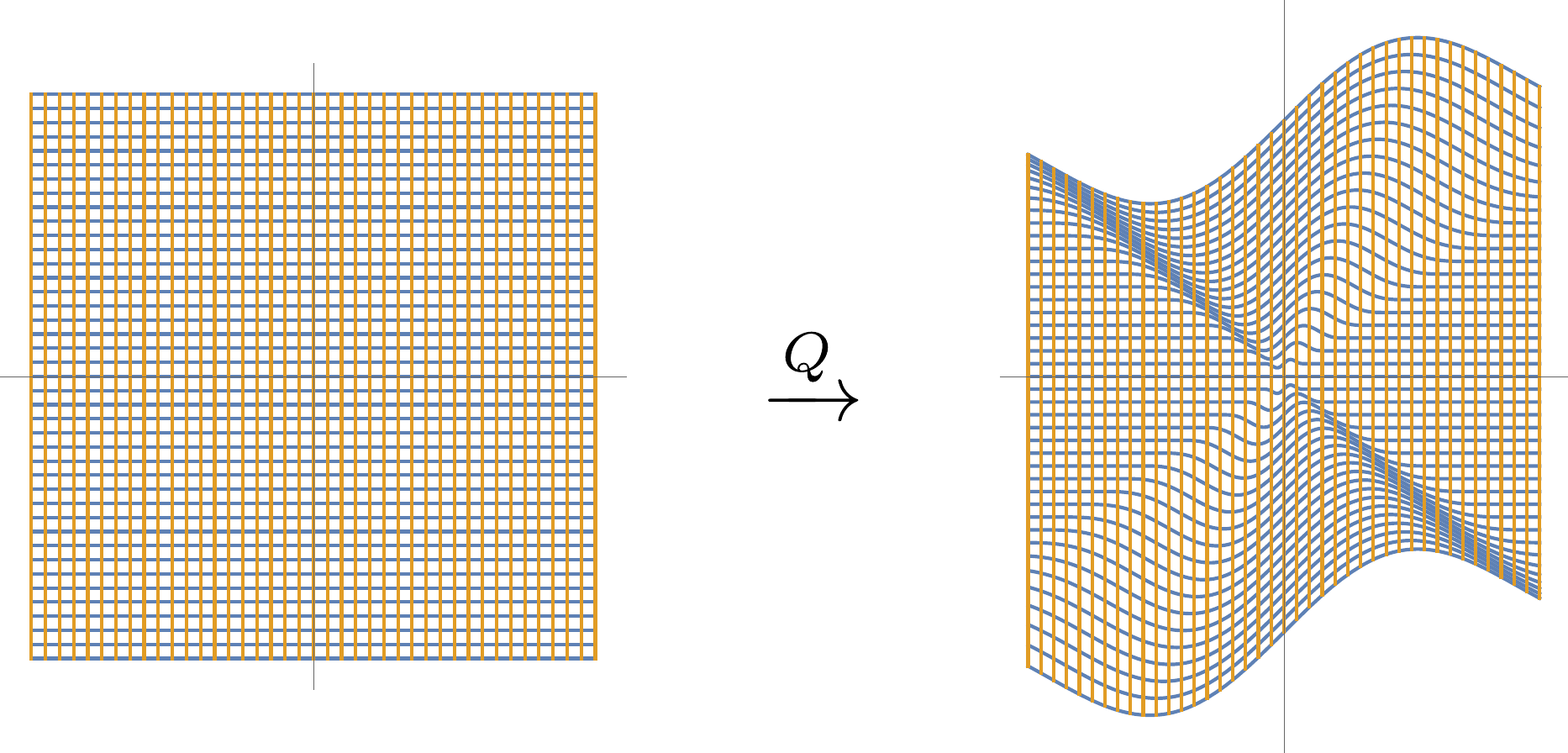}
\caption{The action of $Q$ on a grid near the origin}
\end{figure}

\begin{proof}
First, take a $C^1$ map $q: \mathbb{R} \rightarrow \mathbb{R}^+$ such that $q(0) = 1, q'(0) = 1$, and $q(t) = 1$ whenever $t \notin [-1,1]$. In addition,  pick such a $q$ so that $q(t) > t \, q'(t)$ for all $t$.  

We then define $Q: \mathbb{R}^2 \rightarrow \mathbb{R}^2$ by
$$Q(x,y) = \left\{ \begin{array}{ll} (x , y \,  q(x/y) ) & \mbox{if } y \neq 0 \\  (x, y) & \mbox{if } y =0 \end{array} \right.$$

As before, we denote elements of $\mathbb{P}T\mathbb{R}^2$ by $(p,  \ell)$ where $p \in \mathbb{R}^2$ is a point and $\ell \in \mathbb{P}T_p \mathbb{R}^2$.  Let $\ell_m$ denote a line with slope $m$ in $\mathbb{P}T\mathbb{R}^2$.
We now consider the sequence $\{((0,1/n), \ell_0)\}_{n \in \mathbb{N}}$ of horizontal lines converging to $((0,0), \ell_0)$ in $\mathbb{P}T_{(0,0)}\mathbb{R}^2$.  After applying the induced map $\bar{d}Q$ to this sequence,  the image will be the sequence $\{((0, 1/n), \ell_1)\}_{n \in \mathbb{N}}$. But $\bar{d}Q$ is the identity on $((0,0), \ell_0)$.  Since the sequence $\{((0, 1/n), \ell_1)\}_{n \in \mathbb{N}}$ does not converge to $((0,0), \ell_0)$, then $\bar{d}Q$ is not continuous.  
\end{proof}

\medskip

\noindent \textit{Example $(iv)$: Countably many non-differentiable points.  } A natural question to consider is whether elements of $\thomeo(S)$ can have infinitely many non-differentiable points that accumulate on the surface.  The following result related to part $(iv)$ from Theorem~\ref{propexamples} gives an answer to this question. 

\begin{lemma}\label{lemex4}
There exists an element $P \in\thomeo \!\left(\mathbb{R}^2\right){\setminus}\diff^1\!\left(\mathbb{R}^2\right)$ such that $P$ is not differentiable at the points $\left(2^{-n}, 0 \right)$, for all $n \in \mathbb{N}$. 
\end{lemma}

\begin{proof}
We start with Le Roux--Wolff's \cite{LRW} example from $\thomeo \!\left(\mathbb{R}^2\right){\setminus}\diff^1\!\left(\mathbb{R}^2\right)$: 
$$P_0(x,y)= \left\{ \begin{array}{ll} (x,x \, p_0(y/x)) & \mbox{if } x \neq 0 \\  (x, y) & \mbox{if } x =0 \end{array} \right.$$
 where $p_0: \mathbb{R}\rightarrow \mathbb{R}$ is a diffeomorphism such that $p_0(t) = t$ when $t$ is outside $[1/2, 2]$, but $p_0(1) \neq 1$.  Our goal is to make a function that has similar behavior at a sequence of points that converge to the origin, while becoming increasingly nicer as the points make the approach.  To that end, we define a sequence of functions
 $$p_k(t) = \frac{1}{2^k} p_0(t)+\left( 1- \frac{1}{2^k} \right)t$$ Note that this family approaches the identity as $k\rightarrow \infty.$ Let $\varphi_0 :\mathbb{R}^2 \rightarrow [0,1]$ be a smooth bump function such that $\varphi_0(x,y) = 1$ when $||(x,y)||<\frac1 2$ and $\varphi_0(x,y) = 0$ when $||(x,y)|| \geq 1$.  Similarly, define a sequence of functions 
$$\varphi_k(x,y) = \varphi_0\left(2^{k+2}(x-2^{-k}), 2^{k+2}y \right)$$ that translate and shrink this bump function. 
 We can now define the functions 
$$\rho_n(x, y) = \left(x, \left(x-2^{-n} \right)p_n\left( \frac{ y}{x-2^{-n}} \right) \varphi_n(x,y) + \left(1-\varphi_n(x, y) \right) y \right)$$ 

\begin{figure}[ht]
\centering
\includegraphics[width=119mm]{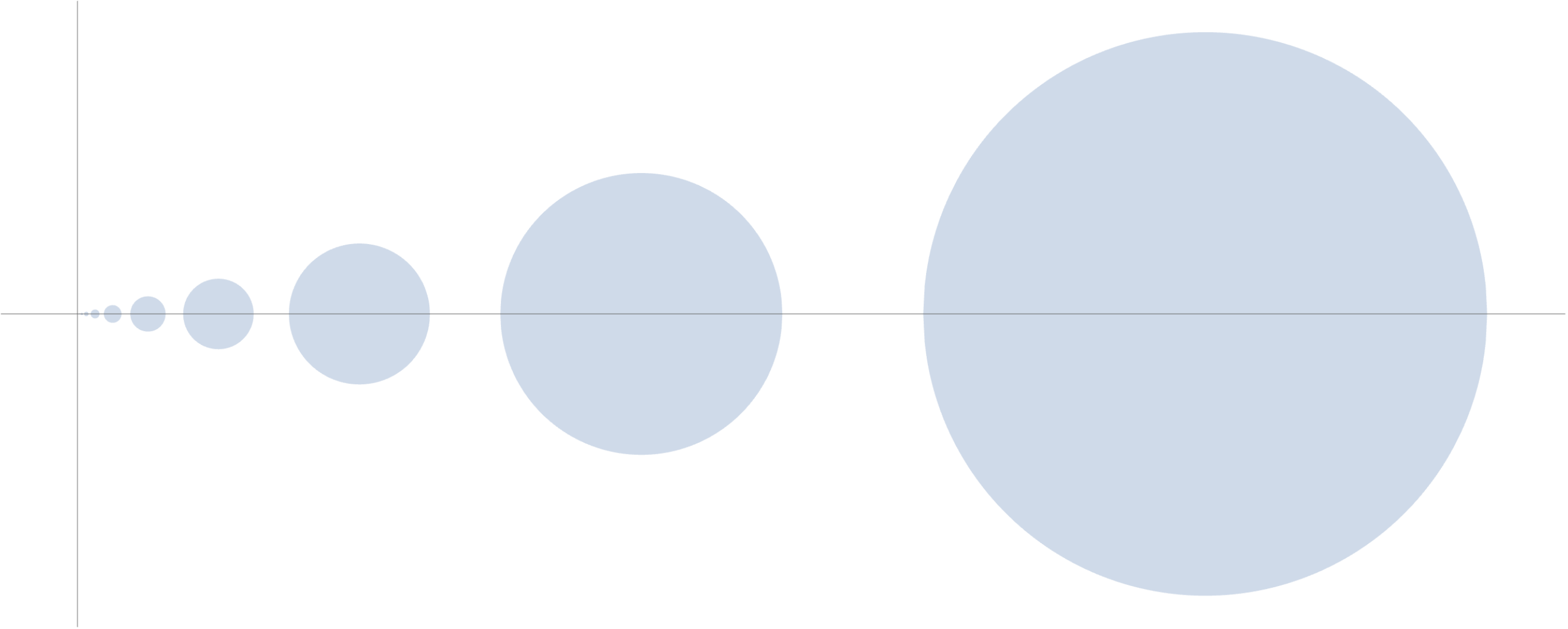}
\caption{The support of $P$}
\end{figure}

Note that each $\rho_n$ has a single non-differentiable point and support contained in the open ball of radius $2^{-(n+2)}$ around the point $(2^{-n},0)$, which means that these supports are all disjoint.  Define $P: \mathbb{R}^2 \rightarrow \mathbb{R}^2$ by 

$$P(x, y) = \left\{ \begin{array}{ll} \rho_n(x, y) & \mbox{when } (x, y) \in \supp(\rho_n) \\
(x, y) & \mbox{otherwise} \end{array} \right. \\[1em]$$

Any curve that is $C^1$ and passes through only finitely many of the $\supp(\rho_n)$ will continue to be $C^1$ after $P$ is applied since each of the non-differentiable points are isolated.   On the other hand,  it can be shown by direct computation that the image of any $C^1$ curve that passes through infinitely many of the $\supp(\rho_n)$ will still be a $C^1$ curve.  In particular, when the $p_n$ approach the identity,  the distortion on the curve also decreases as it approaches the origin, preserving the continuity of the tangent directions.  Thus $P$ maps every $C^1$ curve to a $C^1$ curve.

Since $P^{-1}$ can be described identically, then $P^{-1}$ also maps every $C^1$ curve to a $C^1$ curve. Thus $P$ is an element of $\thomeo \!\left(\mathbb{R}^2\right)$ that is not differentiable at the points  $(2^{-n},0)$.  
\end{proof}

 So the set of non-differentiable points for elements of $\thomeo(S)$ does not have to be a closed set nor finite.
In addition,  this construction also works when point $x_{\infty} = (0,0)$ is non-differentiable and has support on the region of $\mathbb{R}^2$ where $|y|>|x|$.

\subsection{The proof of Theorem \ref{propexamples}}

\begin{proof}
We can add the non-differentiable points in $\mathbb{R}^2$ from the examples to an identity map on any surface $S$ using smooth local charts. 
More precisely,  let $\varphi_i: U_i \subset S \rightarrow \mathbb{R}^2$ for $i\in\{1, 2, 3, 4\}$ be smooth charts such that the $U_i$ are disjoint.
We define the map $F: S \rightarrow S$ by 
$$F(p) = \left\{ \begin{array}{ll} (\varphi_1^{-1} \circ G \circ \varphi_1)(p) & \mbox{ if } p \in U_1 \\
(\varphi_2^{-1} \circ H \circ \varphi_2)(p) & \mbox{ if } p \in U_2 \\
(\varphi_3^{-1} \circ Q \circ \varphi_3)(p) & \mbox{ if } p \in U_3 \\
(\varphi_4^{-1} \circ P \circ \varphi_4)(p) & \mbox{ if } p \in U_4 \\
p & \mbox{ otherwise}
\end{array} \right.$$

By Lemmas \ref{lemex1}-\ref{lemex4},  $F \in \thomeo(S){\setminus}\diff^1(S)$ and has the desired properties.
\end{proof}

\section{Connections between curves, sequences, and the tangent map}\label{section-basics}

Before proving Theorem~\ref{theoremconditions}, this section covers some foundational results connecting $C^1$ curves,  sequences of points, and the tangent map $\bar{d}F_p$.

\subsection{Connecting The Dots}\label{sectionconnectdots}

In this section, we prove that it is always possible to draw a $C^1$ curve through a subsequence of points converging to a particular point. The following proposition and argument is inspired by classical work of Rosenthal in the plane \cite{Rosenthal}.

\begin{prop}\label{c1connectthedots}
Given any sequence $\{p_k\}$ converging to $p \in M$, there exists a $C^1$ curve that intersects infinitely many of the $\{p_k\}$. Moreover, if $\{(p_k, \ell_k)\}$ is sequence in $\mathbb{P}TM$ that strongly converges to $\{(p, \ell)\}$, then there exists a $C^1$ curve in $M$ that intersects a subsequence $\{p_{k_m}\}$ such that the tangent line at $p_{k_m}$ is $\ell_{k_m}$. 
\end{prop}

We begin by restricting to the $\mathbb{R}^n$ case:

\begin{lemma}\label{connectthedotssimple}
Let $\{p_k\} \in \mathbb{R}^n$ be a sequence of points that converges to the point $p$. Then there exists an infinite subsequence $\{p_{k_m}\}$ of $\{p_k\}$ and a $C^1$ curve $\gamma$ such that $\gamma$ intersects every point in $\{p_{k_m}\}$.
\end{lemma}

Before beginning the proof of this lemma, we briefly discuss a method to identify $C^1$ curves without needing to find a particular parameterization. Thinking more geometrically, a $C^1$ curve has tangent lines at every point that vary continuously as you travel along the curve.  Work of Burgu\'es--Cuf\'i \cite[Theorem 10]{BC} shows that for simple closed curves,  continuously varying tangent lines is sufficient to recover a global $C^1$-regular parameterization with nonzero tangent vectors. This gives the following:

\begin{remark}\label{C1continuoustangent}
A simple closed curve in $\mathbb{R}^n$ is $C^1$ if and only if it has a continuous tangent line at every point.
\end{remark}

\begin{proof}[Proof of Lemma~\ref{connectthedotssimple}]
First, by potentially taking a subsequence, we may assume that the $\{p_k\}$ are monotonically decreasing in distance to the point $p$. Let $\ell_k$ denote the line in $T_pM$ between $p$ and $p_k$. Since $\mathbb{RP}^{n-1}$ is compact, then the $\{\ell_k\}$ must accumulate onto some line $\ell \in T_pM$. By again taking a subsequence, we may assume that the $\ell_k$ converge monotonically to $\ell$ such that $$d_{\mathbb{RP}^{n-1}}(\ell_{k+1}, \ell) < \frac 1 2 d_{\mathbb{RP}^{n-1}}(\ell_{k}, \ell)$$ 
for any distance measure in $\mathbb{RP}^{n-1}$. By using a rotation around $p$ and a translation, we may assume that $\ell = [1, 0, \ldots, 0]$, $p=(0, \ldots, 0)$, and that distance in $\mathbb{RP}^{n-1}$ near $\ell$ is given by 
$$d_{\mathbb{RP}^{n-1}}([1, x_2, \ldots, x_n], [1, y_2, \ldots, y_n]) = \sqrt{(x_2-y_2)^2+ \cdots + (x_n-y_n)^2}$$
We may take a further subsequence, and possible a reflection, so that  the points $p_k= (x^k_1, x^k_2, \ldots x^k_n)$ and $p_{k+1}=(x^{k+1}_1, x^{k+1}_2, \ldots x^{k+1}_n)$ are such that $$0 < x^{k+1}_1< \frac 1 2 x^k_1$$ Then the line between $p_k$ and $p_{k+1}$ is associated to the line $$\ell'_{k} = \left[ 1, \frac{x^k_2 - x^{k+1}_2}{x^k_1 - x^{k+1}_1}, \ldots, \frac{x^k_n - x^{k+1}_n}{x^k_1 - x^{k+1}_1} \right] \in \mathbb{RP}^{n-1}$$
and the distance 
\begin{flalign*}
d_{\mathbb{RP}^{n-1}}(\ell'_k, \ell) &= \sqrt{\left(\frac{x^k_2 - x^{k+1}_2}{x^k_1 - x^{k+1}_1}\right)^2 + \cdots + \left(\frac{x^k_n - x^{k+1}_n}{x^k_1 - x^{k+1}_1}\right)^2} \\
& \leq \frac{\sqrt{\left(x_2^k\right)^2 + \cdots +  \left(x_n^k\right)^2}}{x_1^k-x_1^{k+1}} + \frac{\sqrt{\left(x_2^{k+1}\right)^2 + \cdots +  \left(x_n^{k+1}\right)^2}}{x_1^k-x_1^{k+1}} \\
& < \frac{2\sqrt{\left(x_2^k\right)^2 + \cdots +  \left(x_n^k\right)^2}}{x_1^k} + \frac{\sqrt{\left(x_2^{k+1}\right)^2 + \cdots +  \left(x_n^{k+1}\right)^2}}{x_1^{k+1}} \\
& = 2 d_{\mathbb{RP}^{n-1}}(\ell_k, \ell) + d_{\mathbb{RP}^{n-1}}(\ell_{k+1}, \ell) 
\end{flalign*}
But since the $\ell_k$ converge to $\ell$, then both of these distances will approach 0 as $k \rightarrow \infty$. 

The next step is to smooth out the concatenation of the straight lines $\ell'_k$ at each $p_k$ to make a $C^1$ curve. To accomplish this, replace each straight line segment between $p_k$ and $p_{k+1}$ by the curve
$$\gamma_k(t) = \left\{ \begin{array}{ll} 
(1-2t)\left( \frac{||p_k-p_{k+1}||}{2 \pi}  \sin \left(2 \pi t\right) \mathbf{e_1} + p_{k+1} \right) \\
 \quad + \; 2t \left( \frac{p_k+p_{k+1}}{2} + (p_k-p_{k+1})\left(t - \frac 1 2 \right) \right) &  \quad t \in \left[0, \frac 1 2 \right] \\ 
(2t-1)\left( \frac{||p_k-p_{k+1}||}{2 \pi}  \sin \left( 2 \pi t
\right) \mathbf{e_1} + p_{k} \right) \\
 \quad + \; (2-2t) \left( \frac{p_k+p_{k+1}}{2} + (p_k-p_{k+1})\left(t - \frac 1 2 \right) \right) &  \quad t \in \left[\frac 1 2, 1 \right] \end{array} \right. $$
We denote by $\gamma$ the concatenation of the $\gamma_k$ for all $k \in \mathbb{N}$. We note that $\gamma$ is a continuous arc since
\begin{flalign*}
\gamma_k(1) & = (0+p_k) + 0 = p_k \\
\gamma_{k-1}(0) & = (0+p_k) + 0 = p_k\\
\\
\lim_{t \rightarrow 1/2^-} \gamma_k(t) & = 0+\left(\frac {p_k+p_{k+1}} 2 + 0 \right) = \frac 1 2 (p_k+p_{k+1}) \\
\lim_{t \rightarrow 1/2^+}\gamma_k(t) & = 0+\left(\frac {p_k+p_{k+1}} 2 + 0 \right) = \frac 1 2 (p_k+p_{k+1}) \\
\end{flalign*}
To ensure that $\gamma$ is $C^1$, we take the derivative to get
$$\gamma_k'(t) = \left\{ \begin{array}{ll} 
4t(p_k-p_{k+1}) - ||p_k-p_{k+1}|| \left((2t-1) \cos(2 \pi t) + \frac 1 \pi \sin(2 \pi t) \right) \mathbf{e_1} &  \quad t \in \left[0, \frac 1 2 \right] \\
4(1-t)(p_k - p_{k+1}) + ||p_k - p_{k+1}|| \left((2t-1) \cos(2 \pi t) + \frac 1 \pi \sin(2 \pi t) \right) \mathbf{e_1} &  \quad t \in \left[ \frac 1 2, 1 \right]
\end{array} \right. $$
Checking all endpoints, we find the following:
\begin{flalign*}
\gamma_k'(1) & = 0 + ||p_k - p_{k+1}|| (1+0) \mathbf{e_1} = ||p_k - p_{k+1}||\mathbf{e_1} \\
\gamma_{k-1}'(0)& = 0 - ||p_k - p_{k+1}|| (-1+0) \mathbf{e_1} = ||p_k - p_{k+1}||\mathbf{e_1}\\
\\
\lim_{t \rightarrow 1/2^-} \gamma_k'(t) & = 2(p_k - p_{k+1}) - ||p_k- p_{k+1} || (0+0) \mathbf{e_1} = 2(p_k - p_{k+1}) \\
\lim_{t \rightarrow 1/2^+}\gamma_k'(t) & = 2(p_k - p_{k+1}) + ||p_k- p_{k+1} || (0+0) \mathbf{e_1} = 2(p_k - p_{k+1})\\
\end{flalign*}
so the first derivatives also agree at each concatenation point. 
Also, since the $d_{\mathbb{RP}^{n-1}}(\ell'_k, \ell) \rightarrow 0$ as $k \rightarrow \infty$, then $$(p_k - p_{k+1}) \rightarrow ||p_k-p_{k+1}||\mathbf{e_1}$$ Thus for large enough $k$, the first coordinate of $(p_k - p_{k+1})$ will be greater than $\frac 1 2 ||p_k - p_{k+1}||$. So by again potentially restricting to a further subsequence, the first coordinate of $\gamma'$ will always be positive, and thus $\gamma$ cannot self-intersect.

Moreover, as $k \rightarrow \infty$, the tangent line of $\gamma'_k(t)$ must also approach $\ell$ continuously. Next, connect $\gamma_1(1)$ to $p$ with an arc disjoint to $\gamma$ that agrees with the tangent lines at the endpoints. Thus, using by Remark~\ref{C1continuoustangent}, we have constructed a simple closed $C^1$ curve that intersects an infinite subsequence of the $p_k$ as desired.
\end{proof}

\begin{cor}\label{connectthedots-ndim}
Let $(p_k, \ell_k)$ be a strongly convergent sequence in $\mathbb{P}T\mathbb{R}^n$. Then there exists an infinite subsequence $\{p_{k_m}\}$ of $\{p_k\}$ and a $C^1$ curve $\gamma$ such that $\gamma$ intersects all the $\{p_{k_m}\}$ and the tangent line at $p_{k_m}$ is $\ell_{k_m}$. 
\end{cor}

\begin{proof}
We get the desired result by replacing $\mathbf{e_1}$ in the definition of $\gamma_k$ by any sequence of unit tangent vectors $\mathbf{u_k}$ that approach $\mathbf{e_1}$. A further subsequence may be required to require that $\mathbf{u_k}$ is close enough to $\mathbf{e_1}$ to ensure that the first coordinate of $\gamma'$ remains positive.
\end{proof}

\begin{figure}[ht]
\centering
\includegraphics[width=4in]{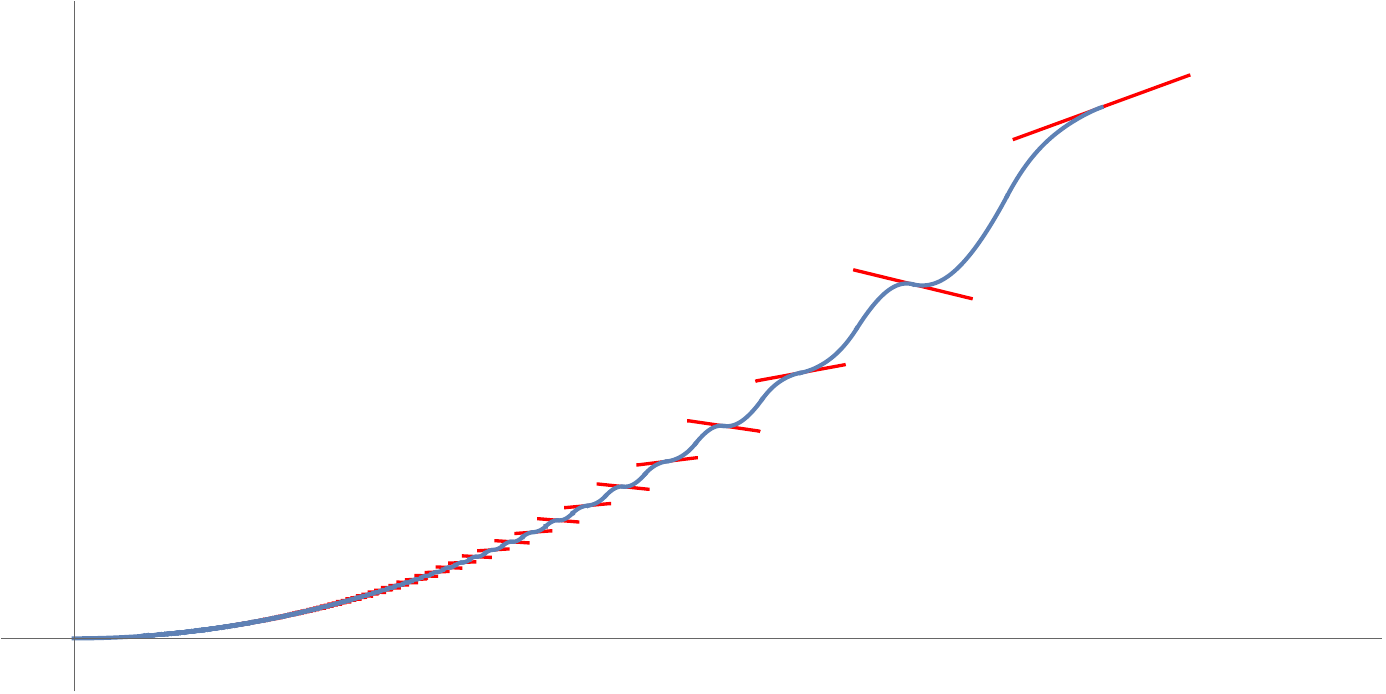}
\caption{A restricted subsequence $\{(p_{k_m}, \ell_{k_m})\} \subset \mathbb{R}^2 \times S^1$ with a $C^1$ arc drawn through the points $p_{k_m}$ with tangent lines $\ell_{k_m}$}
\end{figure}

\begin{proof}[Proof of Proposition~\ref{c1connectthedots}]
By using a smooth local chart around the accumulation point, we may extend the results of both Lemma~\ref{connectthedotssimple} and Corollary~\ref{connectthedots-ndim} directly to any manifold.
\end{proof}

\subsection{Sequences and the tangent map}

In this section, we use Proposition~\ref{c1connectthedots} to give a characterization of $\bar{d}F_p$ being well-defined in terms of sequences in the projective tangent bundle $\mathbb{P}TM$. We start with a definition for $\bar{d}F_p$ that is independent of $C^1$ curves:

\smallskip

\noindent \emph{Existence of maps on tangent lines.} We say that a curve $\alpha:(-\varepsilon, \varepsilon) \rightarrow \mathbb{R}^n$ has a well-defined tangent line at $p=\alpha(0)$ if $$ D_{\alpha}(t) := \lim_{s \rightarrow t^+} \frac{\alpha(t) - \alpha(s)}{||\alpha(t) - \alpha(s)||} = -\lim_{s \rightarrow t^-} \frac{\alpha(t) - \alpha(s)}{||\alpha(t) - \alpha(s)||} $$
exists for $t=0$. Note that if $\ell$ is the tangent line of $\alpha$ at $p$, then $D_\alpha(0)$ is the unit tangent vector $\mathbf{u_{\ell}}$ defined in the introduction. We extend this definition to any manifold by using smooth local charts. 

Let $F \in \homeo(M)$ and $p \in M$. If the image of every locally straight line segment through $p$ has a well-defined tangent line at $F(p)$, then we can define a map $$\bar{d}F_p : \mathbb{P}T_pM \rightarrow \mathbb{P}T_{F(p)}M$$ as the map from straight lines through $p$ to the tangent lines at $F(p)$ of their images.

We can now state the desired characterization:

\begin{prop}\label{convergelineequivwelldefined}
Let $F \in \homeo(M)$ and $p \in M$. Then the following are equivalent:
\begin{enumerate}[label=$(\roman*)$]
\item For any sequence of points $\{p_k\}$ that converges to $p$ along the line $\ell$, there exists a line $\ell'$ such that $\{F(p_k)\}$ converges to $F(p)$ along $\ell'$.
\item The map $\bar{d}F_p$ is well-defined on all curves with a well-defined tangent line at $p$.
\end{enumerate}
\end{prop}

\begin{proof}
For the direction from $(i)$ to $(ii)$, let $\alpha$ and $\beta$ be two curves with the same tangent line $\ell$ at $p$. Then take two sequences $\{p_k\} \subset \alpha$ and $\{q_k\} \subset \beta$ that both converge to $p$. Note that both sequences will converge to $p$ along $\ell$. Now consider the sequence
$$r_m = \left\{ \begin{array}{ll} p_{m/2} & \mbox{ if $m$ is even} \\ q_{(m+1)/2} & \mbox{ if $m$ is odd} \end{array} \right.$$
Note that $\{r_m\}$ is also a sequence that converges to $p$ along $\ell$. By $(i)$, there exists a line $\ell'$ such that $\{F(r_m)\}$ converges to $F(p)$ along $\ell'$. But then any subsequence of $\{F(r_m)\}$ must also converge to $F(p)$ along $\ell'$. Thus both $\{F(p_k)\}$ and $\{F(q_k)\}$ converge to $F(p)$ along $\ell'$. Since these sequences were chosen arbitrarily from $\alpha$ and $\beta$, then $F(\alpha)$ and $F(\beta)$ must both have a tangent line $\ell'$ at $F(p)$. Thus $\bar{d}F_p$ is well-defined on all curves with well-defined tangent lines at $p$.

Next, we argue the direction from $(ii)$ to $(i)$ using the contrapositive. Assume that there is a sequence $\{p_k\}$ that converges to $p$ along a line $\ell$, but the sequence $\{F(p_k)\}$ does not converge to $F(p)$ along any line. Since $S^{n-1}$ is compact, then the values $$\frac{F(p_k)-F(p)}{||F(p_k)-F(p)||}$$ must have an accumulation point, denoted $\mathbf{u}$. Let $\{p_{k_j}\}$ be a subsequence such that $\{F(p_{k_j})\}$ converges to $F(p)$ along the line determine by $\mathbf{u}$, denoted $\ell_{\mathbf{u}}$. Since $\{F(p_k)\}$ cannot converge to $\ell_{\mathbf{u}}$, then there exists some $\varepsilon >0$ and subsequence $\{q_m\}$ of $\{p_k\}$ such that $$\frac{F(q_m)-F(p)}{||F(q_m)-F(p)||} \notin B_{\varepsilon}(\mathbf{u})$$ for all $q_m$, where $B_{\varepsilon}(\mathbf{u})$ is the open ball of radius $\varepsilon$ in $S^{n-1}$ around $\mathbf{u}$. Since $S^{n-1} \setminus B_{\varepsilon}(\mathbf{u})$ is also compact, then $\{F(q_m)\}$ also has an accumulation point $\mathbf{v}$ and a subsequence $\{q_{m_i}\}$ of $\{q_m\}$ such that $\{F(q_{m_i})\}$ that converges to $F(p)$ along the line $\ell_{\mathbf{v}}$. 

By Proposition~\ref{c1connectthedots}, we can find two $C^1$ curves, $\alpha$ through a subsequence of $\{p_{k_j}\}$ and $\beta$ through a subsequence of $\{q_{m_i}\}$. Note that both $\alpha$ and $\beta$ must have a tangent line of $\ell$ at $p$. But, if $F(\alpha)$ has a well-defined tangent line at $F(p)$, then this tangent line must be $\ell_{\mathbf{u}}$, while the tangent line for $F(\beta)$ at $F(p)$, if it exists, must be $\ell_{\mathbf{v}}$. Thus $\bar{d}_pF$ cannot be well-defined for all curves with tangent lines at $p$. 

\end{proof}

\section{Properties of $\thomeo(M)$}\label{sectionforward}

The goal of this section is to prove the following proposition, which is the forward direction of Theorem~\ref{theoremconditions}.

\begin{prop}\label{propthm2forward}
For a smooth manifold $M$, if $F \in \thomeo(M)$,  then $F$ has the following three properties:
\begin{enumerate}[noitemsep,topsep=0pt, label=$(\alph*)$]
\item for every point $p \in M$, there exists a well-defined homeomorphism $\bar{d}F_p$ induced by the action of $F$ on the tangent lines based at $p$,
\item $F$ maps every strongly convergent sequence to a strongly convergent sequence, and
\item $F$ maps every transverse sequence to a transverse sequence.
\end{enumerate}
\end{prop}

We accomplish this proof in three steps, one for each condition. First, we show that if $F$ and $F^{-1}$ both map $C^1$ curves to $C^1$ curves then $\bar{d}F_p$ is well-defined and a homeomorphism.  Mapping $C^1$ curves to $C^1$ curves is then used to show that strongly convergent sequences must be preserved. Finally, we prove property $(c)$,  that transverse sequences are also preserved by $F$.

\subsection{Property $\mathbf{(a)}$: Homeomorphism on $\mathbf{\mathbb{P}T_pM}$}
  The goal of this section is to prove the following lemma related to property $(a)$.

\begin{lemma}\label{projtangentspaceishomeo}
Let $F \in \thomeo(M)$.  Then for all $p \in M$,  there is a well-defined homeomorphism $\bar{d}F_p$ on the tangent lines at $p$ induced by $F$.
\end{lemma}

We start by showing that $\bar{d}F_p$ is well-defined.

\begin{lemma}\label{autfinpreservetangent}
Let $F \in \thomeo(M)$.  Let $\alpha$ and $\beta$ be simple closed $C^1$ curves in $M$ and $p \in \alpha \cap \beta$.  If $\alpha$ and $\beta$ have the same tangent line at $p$ then $F(\alpha)$ and $F(\beta)$ have the same tangent line at $F(p)$.
\end{lemma}

\begin{proof}
Let $a(t)$ and $b(t)$ be $C^1$-regular parameterizations of $\alpha$ and $\beta$, respectively, in a neighborhood of $p$ such that $a(0)=b(0)=p$ and $a'(0) = b'(0) \neq 0$. Then for some $\varepsilon >0$, the arc
$$c(t) = \left\{ \begin{array}{ll} a(t) & \quad t\leq 0 \\ b(t) & \quad t>0 \end{array} \right.$$
is a simple $C^1$-regular arc for $t \in (-\varepsilon, \varepsilon)$ in a neighborhood of $p$. So we may find a simple closed $C^1$ curve $\gamma \in M$ with $c(t)$ as its local parameterization near $p$. 

If $F(\alpha)$ and $F(\beta)$ do not have the same tangent lines at $F(p)$, then $F(\gamma)$ will not have a well-defined tangent line at $F(p)$. From one direction the tangent line matches the tangent line of $F(\alpha)$ at $F(p)$, while from the other direction, the tangent line will agree with the tangent line of $F(\beta)$. But this contradicts the fact that $F(\gamma)$ must be a $C^1$ curve with a tangent line that varies continuously at every point.
\end{proof}

By Lemma~\ref{autfinpreservetangent}, any $F \in \homeo(M)$ that maps every $C^1$ curve to a $C^1$ curve naturally induces a map
$$\bar{d}F_p: \mathbb{P}T_pM \rightarrow \mathbb{P}T_{F(p)}M$$
between the projective tangent spaces. The next step is to show that this map is bijective.

\begin{lemma}\label{tangentlinemapwelldefined}
Let $F \in \thomeo(M)$.  Then for all $p \in M$,   the induced map $\bar{d}F_p$ is a bijection.
\end{lemma}

\begin{proof}
Note that for any $F \in \thomeo(M)$, the inverse map $F^{-1}$ must also be in $\thomeo(M)$.  Thus by Lemma~\ref{autfinpreservetangent},  $F^{-1}$ induces a map $$\bar{d}F^{-1}_{F(p)}: \mathbb{P}T_{F(p)}M \rightarrow \mathbb{P}T_pM$$ Moreover, since $F\circ F^{-1} = F^{-1} \circ F = \id$, then 
$$
\bar{d}F_p \circ \bar{d}F^{-1}_{F(p)}  = \bar{d} \id_{F(p)} = \id  $$
and 
  $$\bar{d}F^{-1}_{F(p)} \circ \bar{d}F_p  = \bar{d} \id_p = \id $$
Thus $\bar{d}F_p$ is a bijection. 
\end{proof}

 \begin{proof}[Proof of Lemma~\ref{projtangentspaceishomeo}]
 Let $\ell_k \in \mathbb{P}T_pM$ be a sequence of lines that converges to the line $\ell \in \mathbb{P}T_pM$. We may lift the $\{\ell_k\}$ to elements $\{\mathbf{u_k}, -\mathbf{u_k}\}$ of $UT_pM$ such that the $\mathbf{u_k}$ converge to the lift $\mathbf{u}$ of $\ell$. 
 
 Note that the bijection $\bar{d}F_p$ also lifts to a well-defined bijection of $UT_pM$, which we will denote by $\widetilde{d}F_p$. Moreover, if $\widetilde{d}F_p$ is continuous, then $\bar{d}F_p$ must also be continuous.
 
 If $\widetilde{d}F_p$ is not continuous at $\mathbf{u}$, then there exists a sequence $\mathbf{u_k}$ converging to $\mathbf{u}$ such that $\widetilde{d}F_p(\mathbf{u_k})$ does not converge to $\widetilde{d}F_p(\mathbf{u_k})$. So there exists $\varepsilon >0$ and a subsequence $\mathbf{u_{k_m}} = \mathbf{v_m}$ such that for all $m$, $$||\widetilde{d}F_p(\mathbf{v_m}) - \widetilde{d}F_p(\mathbf{u})|| > \varepsilon$$
For each $m$, let $\{p^m_i\}_i$ be a sequence of points converging to $p$ such that for all $m$, $$\frac{p^m_i - p}{||p^m_i - p||} = \mathbf{v_m}$$
 Since the straight line through $p$ with unit tangent vector $\mathbf{v_m}$ is a $C^1$ arc, then its image is also a $C^1$ arc, and thus the $\{F(p^m_i)\}$ must converge to $F(p)$ along $\widetilde{d}F_p(\mathbf{v_{m}})$. By taking a subsequence of the $\{p^m_i\}$, we may assume that $$\left|\left|\frac{F(p^m_i) - F(p)}{||F(p^m_i) - F(p)||} - \widetilde{d}F_p(\mathbf{v_{m}}) \right|\right| < \frac 1 {2^i}$$
 
 Now consider the diagonal sequence $\{p^{m}_{m}\}$. Since this sequence converges to $p$, then by Proposition~\ref{c1connectthedots}, there exists a $C^1$ curve $\gamma$ that intersects infinitely many of the $p^m_m$. Moreover,  the tangent line of $\gamma$ at $p$ must be $\ell$ since $$\lim_{m \rightarrow \infty} \frac{p^m_m - p}{||p^m_m-p||} = \lim_{m \rightarrow \infty} \mathbf{v_m}  = \mathbf{u}$$
 Since $F \in \thomeo(M)$, then  $F(\gamma)$ is a $C^1$ curve and the $\{F(p^m_m)\}$ must converge to $F(p)$ along the line determined by $\widetilde{d}F_p(\mathbf{u})$. Thus $$\lim_{m \rightarrow \infty} \frac{F(p^m_m) - F(p)}{||F(p^m_m)-F(p)||} = \widetilde{d}F_p(\mathbf{u})$$
 But 

\begin{flalign*}
& \left|\left|\frac{F(p^m_m) - F(p)}{||F(p^m_m) - F(p)||}  - \widetilde{d}F_p(\mathbf{u}) \right|\right| \\ & \qquad \qquad \geq \left|\left|\widetilde{d}F_p(\mathbf{v_m}) - \widetilde{d}F_p(\mathbf{u}) \right|\right| - \left|\left|\frac{F(p^m_m) - F(p)}{||F(p^m_m) - F(p)||} - \widetilde{d}F_p(\mathbf{v_m}) \right|\right| \\ & \qquad \qquad > \varepsilon - \frac 1 {2^m} > 0 \mbox{ for large } m
\end{flalign*}
which is a contradiction. Thus $\widetilde{d}F_p$ must be a continuous map. By symmetry, the inverse map must also be continuous, and so $\widetilde{d}F_p$ is a homeomorphism. 
 \end{proof}
 
\subsection{Property $\mathbf{(b)}$: Mapping strongly convergent sequences}
In this section, we prove the following lemma, showing  property $(b)$ from Proposition~\ref{propthm2forward} and Theorem~\ref{theoremconditions}.

\begin{lemma}\label{preservestrong}
Let $F \in \thomeo(M)$. Suppose that $\{(p_k, \ell_k)\}$ is sequence that strongly converges to $\{(p, \ell)\}$. Then $\{(F(p_k), \bar{d}F_{p_k}(\ell_k))\}$ strongly converges to $\{(F(p),\bar{d}F_p(\ell))\}$.
\end{lemma}

\begin{proof}
By Proposition~\ref{convergelineequivwelldefined} and Lemma~\ref{projtangentspaceishomeo}, we know that $\{F(p_k)\}$ must converge to $F(p)$ along $\bar{d}F_p(\ell)$. For a contradiction, we assume that $\{(F(p_k), \bar{d}F_{p_k}(\ell_k))\}$ does not strongly converge to $\{(F(p),\bar{d}F_p(\ell)\}$. Thus $\{\bar{d}F_{p_k}(\ell_k)\}$ does not converge to $\bar{d}F_p(\ell)$. So there is a subsequence $\{\bar{d}F_{p_{k_j}}(\ell_{k_j}))\}$ and some $\varepsilon >0$ such that $$\bar{d}F_{p_{k_j}}(\ell_{k_j}) \notin B_{\varepsilon}(\bar{d}F_p(\ell))$$ for all $j$. 

By Proposition~\ref{c1connectthedots}, we can find a $C^1$ curve $\alpha$ through a subsequence $\{(q_m, \ell_m)\}$ of $\{(p_{k_j}, \ell_{k_j})\}$. So $\alpha$ has a tangent line of $\ell$ at $p$ and thus $F(\alpha)$ has a tangent line of $\bar{d}_pF(\ell)$ at $F(p)$. But $F(\alpha)$ contains a sequence of tangent lines $\bar{d}F_{q_m}(\ell_m)$ at a sequence of points $F(q_m)$ converging to the point $F(p)$. Thus $F(\alpha)$ does not have a continuous tangent line and is not a $C^1$ curve, contradicting the assumption that $F \in \thomeo(M)$.
\end{proof}

\subsection{Property $\mathbf{(c)}$: Mapping transverse sequences}

In this section, we prove the following lemma, showing  property $(c)$ from Proposition~\ref{propthm2forward} and Theorem~\ref{theoremconditions}.

\begin{lemma}\label{preservetransverse}
Let $F \in \thomeo(M)$. Suppose that $\{(p_k, \ell_k)\}$ is a transverse sequence that converges to a point $p$ along $\ell$. Then $\{(F(p_k), \bar{d}F_{p_k}(\ell_k))\}$ is a transverse sequence that converges to the point $F(p)$ along $\bar{d}F_p(\ell)$.
\end{lemma}

\begin{proof} 
Let $F \in \thomeo(M)$.  We proceed by contradiction.  Assume that we have a transverse sequence $\{(p_k, \ell_k)\}$ that converges to $p$ along $\ell$ such that $\{(F(p_k), \bar{d}F_{p_k}(\ell_k))\}$ is not a transverse sequence.  

Since $F$ is a homeomorphism, then $\{F(p_k)\}$ must converge to $F(p)$. Since we know that $\bar{d}F_p$ is well-defined for all $C^1$ curves through $p$, then by Proposition~\ref{convergelineequivwelldefined} we know that $\{p_k\}$ converges to $F(p)$ along the line $\bar{d}F_p(\ell)$.

For $\{(F(p_k), \bar{d}F_{p_k}(\ell_k))\}$ to not be a transverse sequence,  it must have a subsequence that converges to $(F(p), \bar{d}F_p(\ell))$.  By Proposition~\ref{c1connectthedots}, there is a $C^1$ curve $\alpha$ through a further subsequence $\{(F(p_{k_{m}}), \bar{d}F_{p_{k_m}}(\ell_{k_{m}}))\}$ such that the tangent line of $\alpha$ at each $F(p_{k_{m}})$ is $\bar{d}F_p(\ell_{k_{m}})$.  Thus $F^{-1}(\alpha)$ must be a curve through the points $p_{k_{m}}$ with tangent lines $\ell_{k_{m}}$.  Since this is a subsequence of the transverse sequence $\{(p_n, \ell_n)\}$,  these tangent lines cannot converge.  So $F^{-1}(\alpha)$ cannot be a $C^1$ curve.  But this contradicts the assumption that $F \in \thomeo(M)$.  
\end{proof}

\begin{proof}[Proof of Proposition~\ref{propthm2forward}]
The desired properties are given by Lemmas~\ref{projtangentspaceishomeo},  \ref{preservestrong} and \ref{preservetransverse}, respectively.
\end{proof}

\section{Recovering $\thomeo(M)$}\label{sectionreverse}

The goal of this section is to show that the properties of  Theorem~\ref{theoremconditions} are sufficient to guarantee that both $F$ and $F^{-1}$ must send $C^1$ curves to $C^1$ curves. More precisely, we prove the following: 

\newpage

\begin{prop}\label{propthm2reverse}
For a smooth manifold $M$, if $F \in \homeo(M)$ has the following three properties:
\begin{enumerate}[noitemsep,topsep=0pt, label=$(\alph*)$]
\item for every point $p \in M$, there exists a well-defined homeomorphism $\bar{d}F_p$ induced by the action of $F$ on the tangent lines based at $p$
\item $F$ maps every strongly convergent sequence to a strongly convergent sequence, and
\item $F$ maps every transverse sequence to a transverse sequence,
\end{enumerate}
then $F \in \thomeo(M)$.
\end{prop}

We complete this proof with two straightforward lemmas. The first uses a subset of the above properties to show that $F$ maps $C^1$ curves to $C^1$ curves. The second uses a different subset to get that $F^{-1}$ also maps $C^1$ curves to $C^1$ curves.

\subsection{The forward map}
We begin by showing that the image of every $C^1$ curve under the map $F$ is also a $C^1$ curve. This result only requires a subset of the properties from Proposition~\ref{propthm2reverse}, given in the following lemma:
\begin{lemma}\label{forwardC1}
If $F \in \homeo(M)$ has the following properties:
\begin{enumerate}[noitemsep,topsep=0pt, label=$(\roman*)$]
\item for every point $p \in M$, there exists a well-defined map $\bar{d}F_p$ induced by the action of $F$ on the tangent lines based at $p$ and
\item $F$ maps every strongly convergent sequence to a strongly convergent sequence,
\end{enumerate}
then $F$ maps every $C^1$ curve to a $C^1$ curve.
\end{lemma}

\begin{proof}
Let $\alpha$ be a $C^1$ curve in $M$. Since $F$ is a homeomorphism, then $F(\alpha)$ must be a simple closed curve. Let $\{(p_k, \ell_k)\}$ be a sequence of points and tangent lines of $\alpha$ converging to a point $p$ along the line $\ell$. Since $\alpha$ is a $C^1$ curve, then $\{(p_k, \ell_k)\}$ is a strongly convergent sequence. By property $(ii)$, $\{(F(p_k), \bar{d}F_p(\ell_k))\}$ is also a strongly convergent sequence converging to $(F(p), \ell')$ for some $\ell' \in \mathbb{P}T_{F(p)}M$. 

In particular, note that this implies that $\{F(p_k)\}$ converges to $F(p)$ along $\ell'$. Moreover, by Proposition~\ref{convergelineequivwelldefined}, we have that $\bar{d}F_p(\ell) = \ell'$. Since the $\{p_k\}$ is an arbitrary sequence that converges to $p$, then for any sequence of points $\{x_k\}$ in $F(\alpha)$ converging to $F(p)$, we may choose $p_k = F^{-1}(x_k)$. Thus every sequence of points from $F(\alpha)$ converging to $F(p)$ converges along the line $\ell'$. So $F(\alpha)$ has a well-defined tangent line at $F(p)$. This same argument can be applied to every point on $F(\alpha)$ to show that $\bar{d}F_p(\ell_k)$ is the tangent line of $F(\alpha)$ at the point $F(p_k)$. Since these tangent lines converge continuously to the tangent line at $F(p)$, then by Remark~\ref{C1continuoustangent}, $F(\alpha)$ is a $C^1$ curve.
\end{proof}

\subsection{The inverse map}

We now give the required properties to ensure that the inverse map $F^{-1}$ also maps $C^1$ curves to $C^1$ curves.

\begin{lemma}\label{inverseC1}
If $F \in \homeo(M)$ has the following properties:
\begin{enumerate}[noitemsep,topsep=0pt, label=$(\roman*)$]
\item $\bar{d}F_p$ is a well-defined homeomorphism at every $p \in M$ and 
\item $F$ maps every transverse sequence to a transverse sequence,
\end{enumerate}
then $F^{-1}$ maps every $C^1$ curve to a $C^1$ curve.
\end{lemma}

\begin{proof}
We begin by using part $(i)$ to show that the inverse image of any $C^1$ curve has well-defined tangent line at every point. We then use part $(ii)$ to show that this tangent line must vary continuously. 

Let $\alpha$ be a $C^1$ curve in $M$ that passes through the point $p$. Since $F$ is a homeomorphism, then $F^{-1}(\alpha)$ is a simple closed curve in $M$. Let $\ell$ be the tangent line of $\alpha$ at $p$. Suppose for a contradiction, that $F^{-1}(\alpha)$ does not have a well-defined tangent line at $F^{-1}(p)$. Then there exist sequences $\{q_k\}$ and $\{r_m\}$ of $F^{-1}(\alpha)$, both converging to $F^{-1}(p)$ along different lines $\ell_q$ and $\ell_r$, respectively. By Proposition~\ref{c1connectthedots}, there exist two $C^1$ curves, $\beta$ through a subsequence of $\{q_k\}$ and $\gamma$ through a subsequence of $\{r_m\}$. Thus the tangent line of $\beta$ at $F^{-1}(p)$ is $\ell_q$ while the tangent line of $\gamma$ at $F^{-1}(p)$ is $\ell_r$. 
Since $\bar{d}_{F^{-1}(p)}F$ is a well-defined homeomorphism, then $\{F(\beta)\}$ and $\{F(\gamma)\}$ have well-defined distinct tangent lines $\bar{d}F_{F^{-1}(p)}(\ell_q)$ and $\bar{d}F_{F^{-1}(p)}(\ell_r)$. But both $\{F(q_k)\}$ and $\{F(r_m)\}$ are also sequences in $\alpha$ converging to the point $p$, contradicting the fact that $\alpha$ has a well-defined tangent line $\ell$ at $p$. Thus $F^{-1}(\alpha)$ must have a well-defined tangent line at $F^{-1}(p)$. 

We now show that the tangent lines of $F^{-1}(\alpha)$ must vary continuously. If not, then there exists a transverse sequence $\{(p_k, \ell_k)\}$ with $p_k \in F^{-1}(\alpha)$ converging along $\bar{d}F^{-1}_p(\ell)$ to the point $F^{-1}(p)$ with tangent line $\ell_k$ of $F^{-1}(\alpha)$ at each point $p_k$. By part $(ii)$, the sequence $\{(F(p_k), \bar{d}F_{F^{-1}(p)}(\ell_k))\}$ must also be a transverse sequence. But this sequence is exactly a strongly convergent sequence for $\alpha$ and cannot be transverse. Thus the tangent lines of $F^{-1}(\alpha)$ must vary continuously and so $F^{-1}(\alpha)$ is a $C^1$ curve. 
\end{proof}

\subsection{Finishing the proofs}

We are now able to tie all of the previous results together to complete the proof of our main result.

\begin{proof}[Proof of Proposition~\ref{propthm2reverse}]
The proposition follows directly from Lemmas~\ref{forwardC1} and \ref{inverseC1}.
\end{proof}

\begin{proof}[Proof of Theorem~\ref{theoremconditions}]
Proposition~\ref{propthm2forward} proves the forward direction of the theorem.  The reverse direction follows from Proposition~\ref{propthm2reverse}.
\end{proof}

\bibliographystyle{alpha}
\bibliography{mybib3}

\end{document}